\documentclass[11pt]{amsart}

\usepackage{amsfonts, amstext, amsmath, amsthm, amscd, amssymb}
\usepackage{graphicx, color, epsfig, subfigure, wrapfig, overpic}

\let\oldmarginpar\marginpar
\renewcommand\marginpar[1]{\oldmarginpar[\raggedleft\footnotesize #1]%
{\raggedright\footnotesize #1}}

\newcommand{\Q}{\mathbb{Q}}
\newcommand{\Z}{\mathbb{Z}}

\newcommand{\bfloor}[1]{\left\lfloor #1\right\rfloor}

\theoremstyle{plain}
\newtheorem{theorem}{Theorem}[section]

\newtheorem{lemma}[theorem]{Lemma}
\newtheorem{prop}[theorem]{Proposition}

\newtheorem{conjecture}{Conjecture}[section]

\theoremstyle{definition}

\newtheorem{remark}[theorem]{Remark}

\newtheorem*{namedtheorem}{\theoremname}
\newcommand{\theoremname}{testing}

\title[A note on quasi-alternating Montesinos links]{A note on quasi-alternating Montesinos links}

\author[A.\ Champanerkar]{Abhijit Champanerkar}
\address{ Department of Mathematics, College of Staten
Island, City University of New York, Staten Island, New
York 10314 - and - Mathematics Program, The Graduate Center, City
University of New York, 365 Fifth Avenue, New York, New York
10016 }
\email{abhijit@math.csi.cuny.edu}
\thanks{The first author gratefully acknowledges support by the Simons Foundation and PSC-CUNY. The second author gratefully acknowledges support by PSC-CUNY} 
\author[P.\ Ording]{Philip Ording}
\address{Department of Mathematics, Medgar Evers College, City
University of New York, 1650 Bedford Ave,
Brooklyn, NY 11225}

\email{pording@mec.cuny.edu}

\begin{document}

\maketitle

\begin{abstract}
  Quasi-alternating links are a generalization of alternating links.
  They are homologically thin for both Khovanov homology and knot
  Floer homology.  Recent work of Greene and joint work of the first
  author with Kofman resulted in the classification of
  quasi-alternating pretzel links in terms of their integer tassel
  parameters.  Replacing tassels by rational tangles generalizes
  pretzel links to Montesinos links. 
  In this paper we establish conditions on the rational parameters of
  a Montesinos link to be quasi-alternating.  Using recent results on
  left-orderable groups and Heegaard Floer L-spaces, we also establish
  conditions on the rational parameters of a Montesinos link to be
  non-quasi-alternating. We discuss examples which are not covered by
  the above results.
 \end{abstract}

\section{Introduction}
The set $\mathcal{Q}$ of quasi-alternating links was defined by
Ozsv\'ath and Szab\'o \cite{Osvath-Szabo:DoubleCovers} as the smallest
set of links satisfying the following:\\
\begin{minipage}{3.8in}
\begin{itemize}
\item the unknot is in $\mathcal{Q}$
\item if link $\mathcal{L}$ has a diagram $L$ with a crossing $c$ such that 
\begin{enumerate}
\item both smoothings of $c$, $L_0$ and $L_\infty$ are in $\mathcal{Q}$
\item $\det(L_0)\neq 0 \neq \det(L_\infty)$
\item $\det(L)=\det(L_0)+\det(L_\infty)$
\end{enumerate}
then $\mathcal{L}$ is in $\mathcal{Q}$.
\end{itemize}
\end{minipage}
\begin{minipage}{2.4in}
\begin{center}
\includegraphics[width=2.4 in]{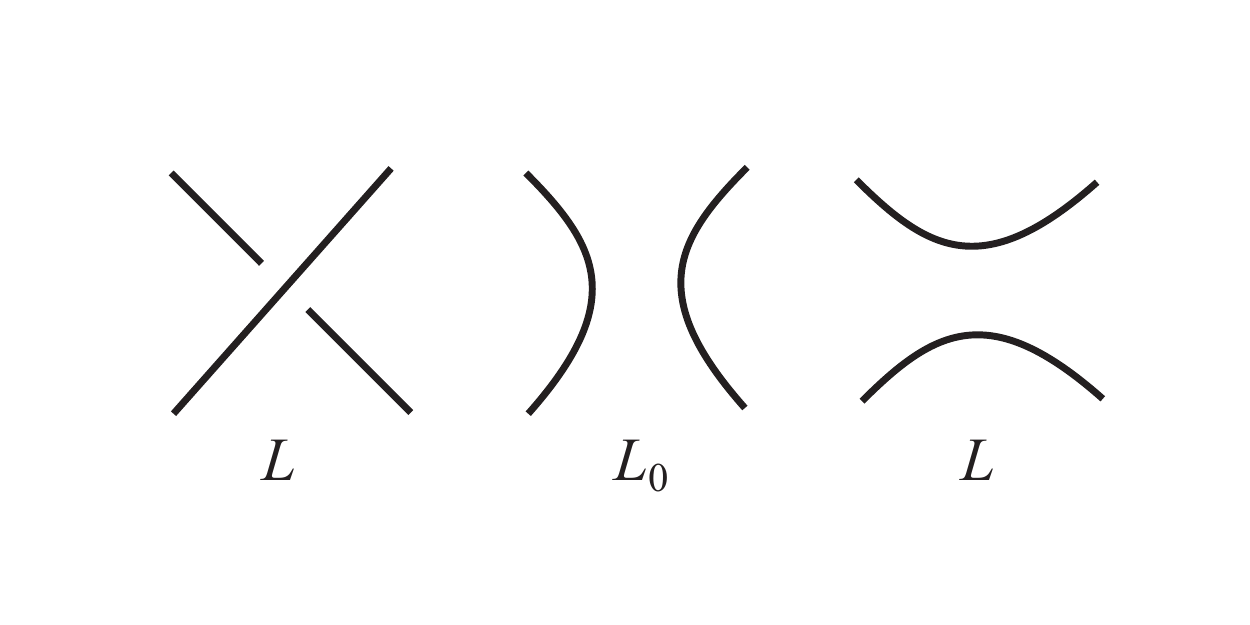}
\end{center}
\end{minipage}\\

The set $\mathcal{Q}$ includes the class of non-split alternating links. Like alternating links, quasi-alternating links are
homologically thin for both Khovanov homology and knot Floer
homology  \cite{Manolescu}. 
The branched double covers of quasi-alternating links are
L-spaces \cite{Osvath-Szabo:DoubleCovers}. These properties make
$\mathcal{Q}$ an interesting class to study from the knot homological
point of view. However the recursive definition makes it difficult to
decide whether a given knot or link is quasi-alternating.

The first author and Kofman showed that the
quasi-alternating property is preserved by replacing a
quasi-alternating crossing by any rational tangle extending the
crossing, and they used it to give a sufficient condition for pretzel links
to be quasi-alternating \cite{Champanerkar}. 
Subsequently Greene showed that this
condition was necessary and provided the first examples of
homologically thin, non-quasi-alternating links \cite{Greene}.  Results in
\cite{Champanerkar} and \cite{Greene} provide a complete
classification of quasi-alternating pretzel links. 

Using the structure and symmetry of Montesinos links and their
determinants, we generalize the sufficient conditions given in
\cite{Champanerkar} and \cite{Greene} to provide a sufficient
condition for Montesinos links to be quasi-alternating, in terms of
their rational parameters (Theorem \ref{thm:qam}).  Using recent
results on left-orderable groups, Heegaard Floer L-spaces and branched
double covers of Montesinos links we also obtain conditions on the
rational parameters of a Montesinos link to be non-quasi-alternating
(Theorem \ref{thm:nqam}). Furthermore we discuss families of examples
which are not covered by the above results.  
Our results include all known classes of quasi-alternating links
appearing in \cite{Champanerkar}, \cite{Greene} and \cite{Widmer}. See
also the recent preprint of Qazaqzeh, Chbili, and Qublan
\cite{Khaled}.\footnote{The results in this paper were obtained
  independently of \cite{Khaled}.}
Watson gives an iterative construction for obtaining every
quasi-alternating Montesinos link using surgery on a strongly
invertible L-space knot \cite{Watson}. It is an interesting problem to
determine the relation between Watson's construction and the
conditions in Theorem \ref{thm:qam}.

This paper is organized as follows: Section 2 defines Montesinos links
and related notation, Section 3 proves results about the structure and
symmetry of Montesinos links, Section 4 proves the determinant formula
for Montesinos links, and Sections 5 and 6 prove the main theorems and 
discuss examples.

\subsection*{Acknowledgements} The authors would like to thank Joshua 
Greene, Liam Watson and Steven Boyer for helpful conversations.

\section{Notation}

\subsection{ Fractions} For integers $a_i, 1\leq i \leq m$, $a_1 \neq 0$, let
$[a_m,a_{m-1},\dots,a_1]$ denote the continued fraction
$$[a_m,a_{m-1},\dots,a_1]:= a_m+\cfrac{1}{a_{m-1}+\cfrac{1}{\ddots+\cfrac{1}{a_1}}} \ .
$$
Let
$\displaystyle{t=\frac{\alpha}{\beta}}\in \Q$ with $\alpha,\beta$ relatively prime and $\beta > 0$. The
\emph{floor} of $t$ is $\displaystyle{\left\lfloor t
  \right\rfloor = \frac{\alpha - (\alpha \mod \beta)}{\beta}}$, and
the \emph{fractional part} of $t$ is $\displaystyle{\{t\} = \frac{\alpha \mod \beta}{\beta} < 1}$.  For $t \neq 0$, define $
\displaystyle{\widehat{t} = \frac{1}{\{\frac{1}{t}\} } > 1}$. For example,
if $t=\frac{-29}{9}$ then $\left\lfloor t
  \right\rfloor = -4,\ \{ t \} = \frac{7}{9}, \mathrm{\ and \ } 
\widehat{t}=\frac{29}{20}$. Note that if $t > 1$ then $\widehat{t}=t$. 

\subsection{ Rational tangles} 
 
We follow the original exposition due to Conway \cite{Conway}.
A \emph{tangle} is a portion of a link diagram enclosed by a circle that meets the link in exactly four points. 
The four ends of a tangle are identified with the compass directions NW, NE, SW, SE. 
Given a pair of tangles $s$ and $t$, the \emph{tangle sum}, denoted $s+t$, is formed by joining the NE, SE ends of $s$ to the NW, SW ends, respectively, of $t$. 
The elementary tangles $0,\pm 1, \infty$ are shown in Figure \ref{fig:tangleexample}.
Adding $n$ copies of the tangle $1$ and $\overline{1}=-1$ results in the \emph{integral tangles} $n=1+1+\dots+1$ and $\overline{n}=-n=\overline{1}+\overline{1}+\dots+\overline{1}$, respectively.
The \emph{tangle product}, denoted $st$, is the tangle obtained by first reflecting the diagram of $s$ in the plane through its NW-SE axis and then adding $t$. 
If $a_1,\dots,a_m$ are integral tangles, the tangle $a_1 a_2 \dots a_m:=((\dots(a_1 a_2)a_3\dots a_{m-1})a_m)$ is called a \emph{rational tangle}. See Figure \ref{fig:tangleexample}.
\begin{figure}[htbp]
\begin{center}
\includegraphics{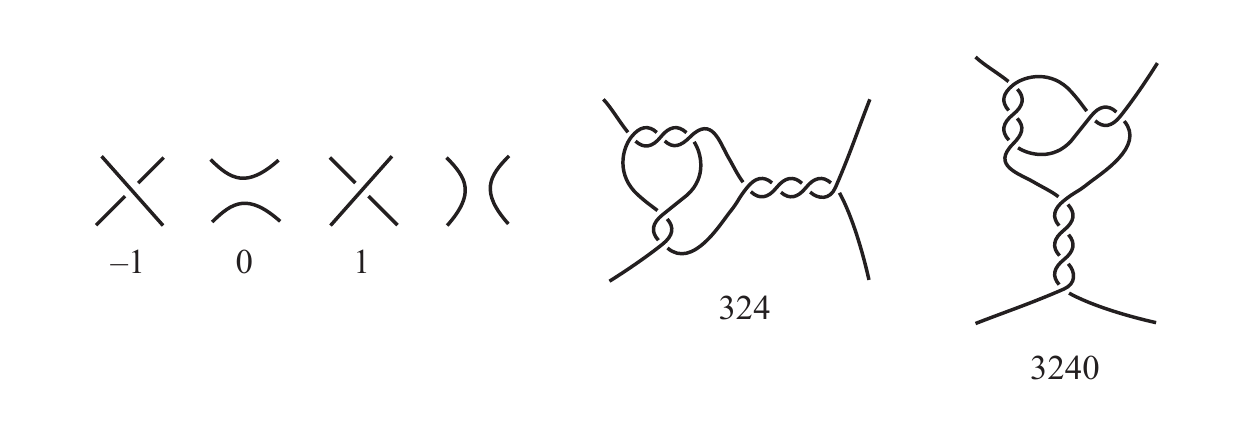}
\caption{Rational tangles are composed of sums and products of elementary tangles.}
\label{fig:tangleexample}
\end{center}
\end{figure}

Conway devised the following correspondence between rational tangles and continued fractions:
Let $t\neq 0, \pm 1$ be a rational number, and let $a_1, \ldots , a_m$ be integers such that $a_1 \geq 2$, $a_k \geq 1$ for $k=2,\ldots, m-1$, and $a_m \geq 0$.  
If $t=[a_m,a_{m-1},\dots,a_1]$, then $t$ corresponds to the positive rational tangle $a_1a_2\ldots a_m$.  
If $t=[-a_m,-a_{m-1},\dots,-a_1]$, then $t$ corresponds to the negative rational tangle $\overline{a_1}\overline{a_2}\ldots
\overline{a_m}$, where $\overline{a}$ denotes $-a$.  
Every rational tangle, except for the elementary tangles $0$, $\pm 1$, and
$\infty$, has a continued fraction expansion as one of the above
\cite{Conway}.
The product of a rational tangle with the zero tangle inverts the associated fraction; if $a_1\dots a_m$ corresponds to $t$, then $a_1\dots a_m0$ corresponds to $1/t$.
Notice that $a_1\dots a_m a_{m+1}$ is equivalent to $a_1\dots a_m0 + a_{m+1}$, which corresponds to the fraction $a_{m+1}+1/t$, and this explains the correspondence between rational tangles and continued fractions.

A \emph{flype} is a tangle equivalence between tangles $1+t$ and $t_h+1$, where $t_h$ is the rotation of tangle $t$ about its horizontal axis.
A \emph{positive flype} is the operation that replaces $t$ by the equivalent tangle $1+t_h+\overline{1}$, and a \emph{negative flype}
results in $\overline{1}+t_h+1$. 
For a rational tangle $t$, the tangle $t_h$ can be seen to be equivalent to $t$ by a sequence of generalized flype moves that transpose tassels above and below the horizontal axis. 

There are two ways to join the free ends of a tangle (without introducing further crossings) to form a link. 
Let $1*t$ denote the \emph{vertical closure} of tangle $t$ obtained by joining the NW end to the NE end and joining the SW end to the SE end.
Joining the NW/SW ends and the NE/SE ends produces the \emph{horizontal closure} of $t$, which is isotopic to $1*\overline{t}0$. See Figure \ref{fig:tangleclosure}. 
A \emph{rational link} is the closure of a rational tangle. 

\begin{figure}[t]
\begin{center}
\includegraphics{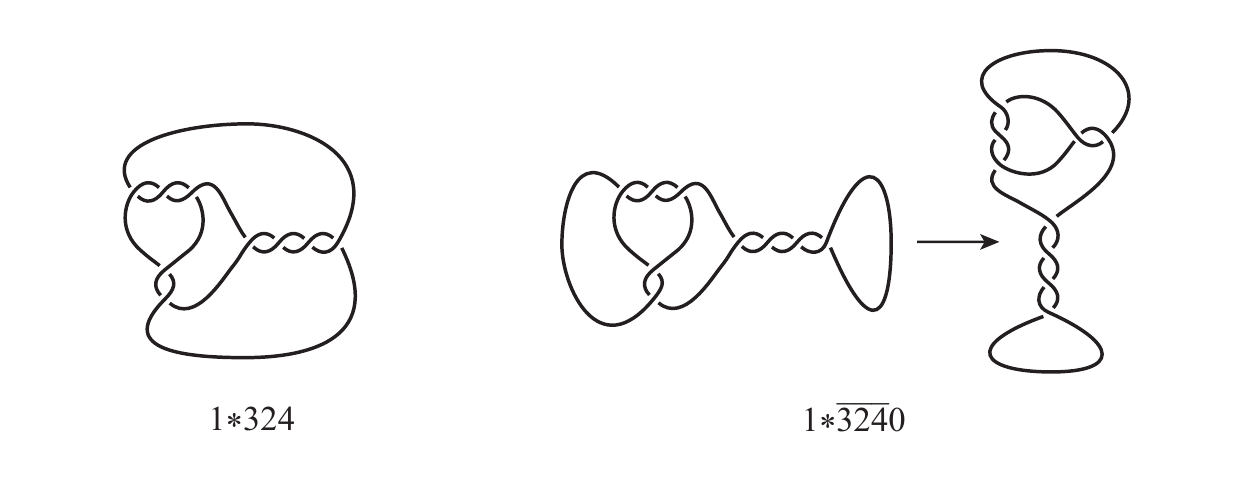}
\caption{The vertical closure (left) and horizontal closure (center) of tangle $324$, the latter of which is equivalent to the vertical closure of $\overline{324}0$ (right).}
\label{fig:tangleclosure}
\end{center}
\end{figure}

\subsection{ Montesinos links } 
Let $t_{i}\neq 0, \pm 1$, for $i=1,\ldots,p$, 
be a rational number with a continued fraction
expansion as above, and let $e$ be an integer.  A \emph{Montesinos
  link} is defined, using Conway notation, as $M(e;t_1, \ldots,
t_p)=1*(e+t_10+\ldots +t_p0)$. See Figure \ref{fig:montesinos}; the
dotted circle labeled $t_i$ contains tangle $t_i0$.  We define
$\displaystyle{\varepsilon= e + \sum_{i=1}^p
  \left\lfloor\frac{1}{t_i}\right\rfloor}$.  Note that this
presentation of Montesinos links differs slightly from that of
Burde-Zieschang \cite{Burde-Zieschang} and the one used by Greene
\cite{Greene} in that the sign of $e$ is reversed.  For example,
Figure \ref{fig:montesinosvar} illustrates the isotopy taking the link
of Figure \ref{fig:montesinos} into the form of a Montesinos link used in 
\cite{Greene}. 
If $t_i$ were $\pm 1$, then the
application of a flype would move the crossing left, where it could be
absorbed by the parameter $e$.

\begin{figure}[h!]
\begin{center}
\includegraphics{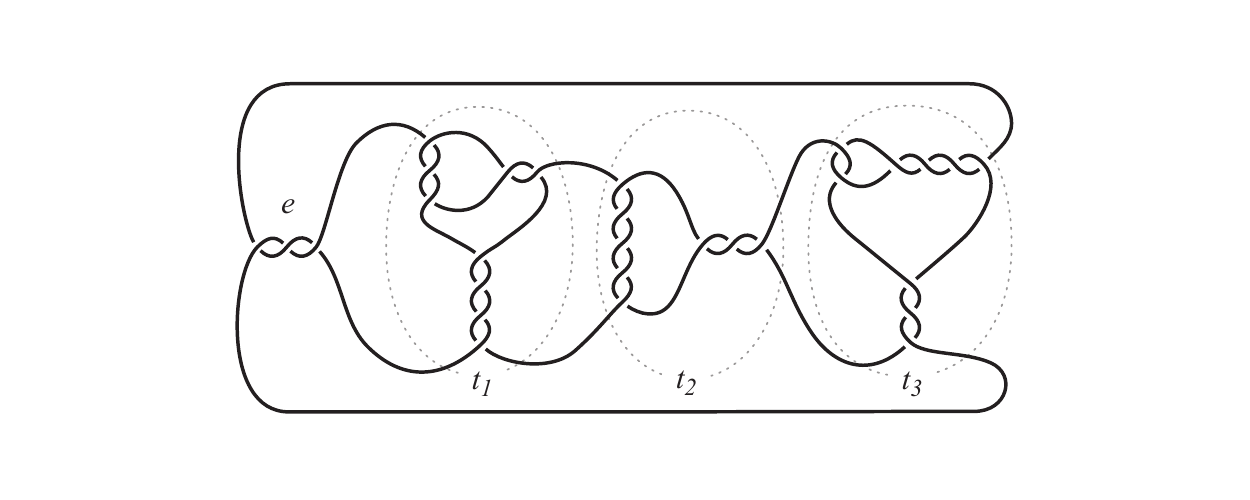}
\caption{Montesinos link $M(3;\frac{31}{7},\frac{5}{16},\frac{-29}{9}) =M(3;324,530,\overline{2}\overline{4}\overline{3})=M(5;\frac{31}{7},\frac{5}{1},\frac{29}{20})$.}
\label{fig:montesinos}
\end{center}
\end{figure}

\begin{figure}[htbp]
\begin{center}
\includegraphics[height=3in]{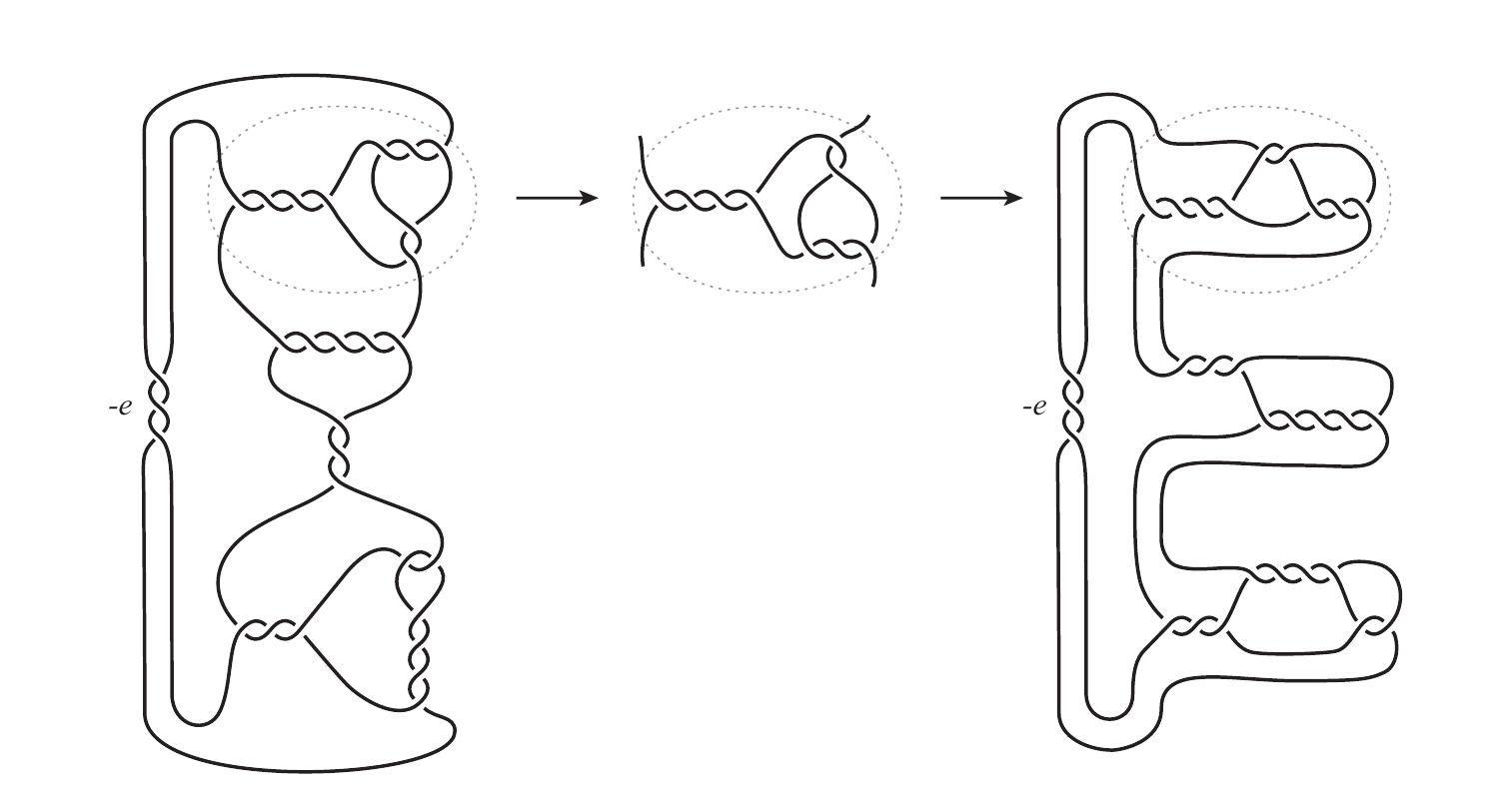}
\caption{Rotate the link of Figure \ref{fig:montesinos} clockwise 90 degrees, isotope the $-e$ crossings to the left side, and apply flypes to put each tangle into braid form.}
\label{fig:montesinosvar}
\end{center}
\end{figure}

\section{Classification of Montesinos links}

Let $L$ be the Montesinos link
$M(e;t_1,\dots,t_p)=M(e;\alpha_1/\beta_1,\ldots,\alpha_p/\beta_p)$. 

\begin{prop}\label{prop:2tangle}
If $p<3$, then $L$ is isotopic to a rational link. 
\end{prop}

\begin{proof}
Let $[a_k,\ldots,a_1]$ and $[b_{\ell},\ldots,b_1]$ be continued fraction expansions of $t_1$ and $t_2$, respectively. 
Applying $e$ flypes to the first tangle moves the $e$ crossings between the two tangles. Applying isotopies and flypes to the tassels $b_\ell, b_{\ell-1},\dots,b_1$, in that order, results in rational tangle form. See Figure \ref{fig:2tanglemontesinos}. The parity of $k+\ell$ determines the appropriate tangle closure. 
It follows that $L$ is isotopic to the rational link $1*t$ if $k+\ell$ is odd and $1*\overline{t0}$ if $k+\ell$ is even, where $t=[b_1,\dots,b_\ell,e,a_k,\dots,a_1]$.
\end{proof}

\begin{figure}[htbp]
\begin{center}
\includegraphics{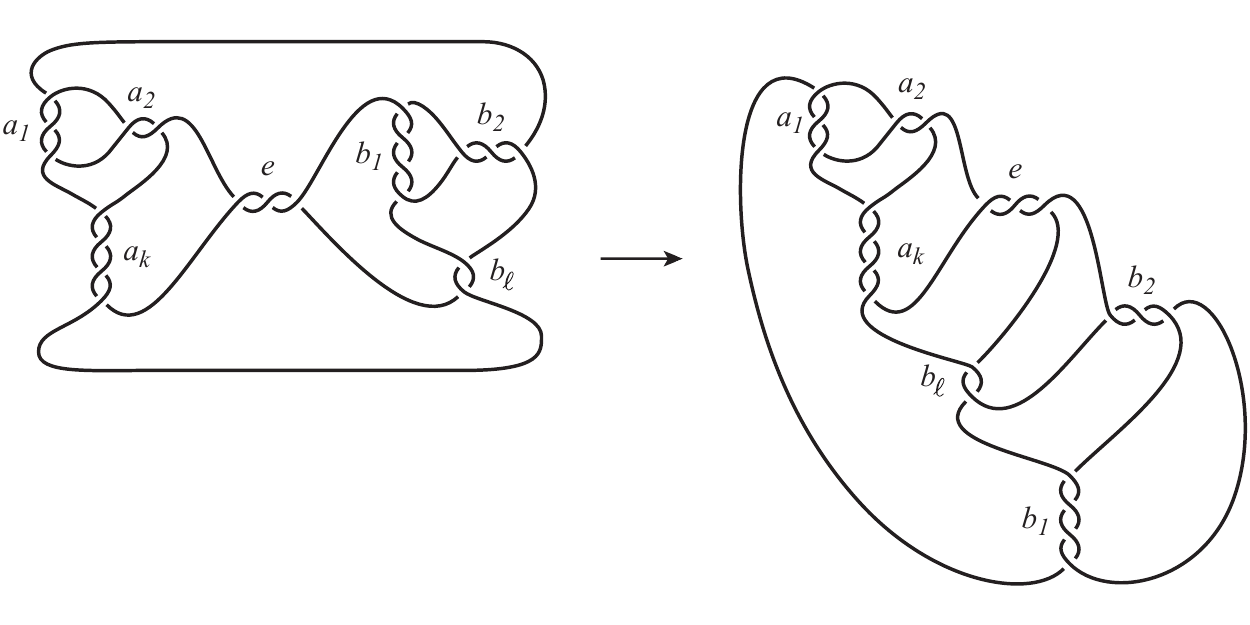}
\caption{Two-tangle Montesinos links are rational links.}
\label{fig:2tanglemontesinos}
\end{center}
\end{figure}

If
$p\geq 3$ 
then $L$ is classified by the rational number
$e+\sum_{i=1}^p \beta_i/\alpha_i$ and the
ordered set of fractions 
$\left(\left\{\beta_1/\alpha_1\right\},\ldots,\left\{\beta_p/\alpha_p\right\}\right)$ 
up to cyclic permutation
and reversal of order \cite{Bonahon} (see also
\cite{Burde-Zieschang}).  It follows that $\varepsilon$ defined above is an
invariant of Montesinos links.

\begin{prop}
  \label{lem:mont-reduce} 
  The Montesinos link $ M(e; t_1, \dots, t_p)$ is isotopic to
  $M(\varepsilon;\widehat{t_1}, \ldots, \widehat{t_p})$.
\end{prop}

\begin{proof}
A rational tangle $t$ is equivalent to the tangle sum of its integral and fractional parts. In particular, a positive rational tangle $a_1\dots a_m$ is equivalent to $a_m + a_1\dots a_{m-1}0$, where $a_m$ and $a_1\dots a_{m-1}0$ correspond to $\lfloor t\rfloor$ and $\{t\}$, respectively. 
A negative rational tangle $\overline{a_1}\dots\overline{a_m}$ is equivalent to $\overline{a_m}+\overline{a_1}\dots\overline{a_{m-1}}0$, and a negative flype results in the equivalent tangle $\overline{1}+(\overline{a_m}+\overline{a_1}\dots\overline{a_{m-1}}0)+1$. But, $(\overline{1}+\overline{a_m})$ and $(\overline{a_1}\dots\overline{a_{m-1}}0+1)$ are equivalent to $\lfloor t\rfloor$ and $\{t\}$.

The proposition follows by applying the above tangle decomposition to each tangle of the Montesinos link :
\begin{align*}
M(e; t_1, \dots, t_p) &= 1*(e + t_10 + \dots + t_p0) \\
& = 1*\left(e + \left(\left\lfloor t_10 \right\rfloor + \{t_10 \}  \right) + \dots + \left(\left\lfloor t_p0 \right\rfloor + \{t_p0 \}  \right) \right) \\
&=1*\left(
\varepsilon + \widehat{t_1}0 + \dots + \widehat{t_p}0
\right) = M(\varepsilon;\widehat{t_1}, \dots, \widehat{t_p}).\qedhere
\end{align*}
\end{proof}

The link $M(\varepsilon;\widehat{t_1}, \ldots, \widehat{t_p})$ is known as the
{\it reduced form} of the Montesinos link $M(e; t_1, \dots, t_p)$.

\begin{lemma}
\label{lem:flype}
Let $t_i=\frac{\alpha}{\beta}$, for some $i$, and let $e$ be any integer. 
\begin{enumerate}
\item (Positive flype) If $t_i > 0$, then $M(e;t_1,\ldots,t_p)=M(e+1;t_1,\ldots,t_{i-1},t_i^f,t_{i+1},\ldots,t_p)$, where 
$\displaystyle{t_i^f=\frac{\alpha}{\beta- \alpha}}$.
\item (Negative flype) If $t_i < 0$, then $M(e;t_1,\ldots,t_p)=M(e-1;t_1,\ldots,t_{i-1},t_i^f,t_{i+1},\ldots,t_p)$, where 
$\displaystyle{t_i^f=\frac{\alpha}{\beta + \alpha}}$.
\end{enumerate}
\end{lemma}

\begin{proof}
Suppose $t_i>0$. 
In Conway notation, 
$$M(e;t_1,\ldots,t_p)=1*(e+t_10+\ldots+t_{i-1}0+t_i0+t_{i+1}0+\dots +t_p0).$$ 
A positive flype of the first $i$ rational tangles $t_1,\ldots,t_i$ results in the equivalent link
$$1*(e+1+(t_10+\ldots+t_{i-1}0+t_i0)_h+\overline{1}+t_{i+1}0+\dots +t_p0).$$
The horizontal rotation of a tangle sum is the sum of the summands horizontally rotated.
This fact and the invariance of rational tangles under horizontal rotation implies that the link is equivalent to 
$$1*(e+1+t_10+\ldots+t_{i-1}0+t_i0+\overline{1}+t_{i+1}0+\dots +t_p0).$$
Furthermore, 
$$t_i0+\overline{1}=(\alpha/\beta)0+\overline{1}=\beta/\alpha-1=(\beta-\alpha)/\alpha=(\alpha/(\beta-\alpha))0=t_i^f0,$$ as required.

Applying a negative flype, the $t_i<0$ case follows similarly.
\end{proof}

\begin{prop}
\label{prop:adq-alt}
  Let $L=M(e;t_1, \ldots, t_p)$ be a Montesinos link
and $\varepsilon$ as above.  
\begin{enumerate}
\item If $\displaystyle{\left|\varepsilon  + \frac{p}{2}\right|>\frac{p}{2}-1}$, then $L$ has an alternating diagram.
\item If $\displaystyle{\left|\varepsilon + \frac{p}{2}\right|<\frac{p}{2}-1}$, then $L$ has a non-alternating and adequate diagram.
\end{enumerate} 
\end{prop}

\begin{proof} $L$ is equivalent to $
    L'=M(\varepsilon;\widehat{t_1}, \ldots, \widehat{t_p})$ by Proposition
  \ref{lem:mont-reduce} above. 
The inequality $\left|\varepsilon + \frac{p}{2}\right|>\frac{p}{2}-1$ implies
  that $\varepsilon\geq 0$ or $\varepsilon\leq -p$.

  If $\varepsilon\geq 0$, then the reduced form $L'$ is alternating since
  the tangles $\widehat{t_i}$ are
  positive for all $i=1,\ldots, p$.  If $\displaystyle{\varepsilon \leq -p}$,
  then applying a positive flype to each of the $p$ tangles of $L'$ as
  in Lemma \ref{lem:flype}, yields an alternating diagram with all
  negative tangles. This proves the first case.

  For the second case, suppose
  $t_i=\alpha_i/\beta_i$, where $\alpha_i>0$
  for all $i=1,\dots,p$.  Then $L$ is equivalent to $L'
    = M\left(\varepsilon; \alpha_1/(\beta_1\mod{\alpha_1}),
      \dots,\alpha_p/(\beta_p\mod{\alpha_p})\right)$.  
Since $\left|\varepsilon +
      \frac{p}{2}\right|<\frac{p}{2}-1$, $ -p+1 < \varepsilon < -1$, 
hence $1 < |\varepsilon| < p-1$.
Applying
  a positive flype to each of the last $m=|\varepsilon|$ tangles
  of $L'$ results in an equivalent link
  $\displaystyle{L''=M\left(0;r_1,\dots,r_n,s_1,\dots,s_m\right)}$,
  where $n=p-m$,
  $\displaystyle{r_i=\frac{\alpha_i}{\beta_i\mod{\alpha_i}}>0}$ for
  $i=1,\dots,n$, and
  $\displaystyle{s_j=\frac{\alpha_j}{(\beta_j\mod{\alpha_j})-\alpha_j}<0}$
  for $j=1, \dots, m$.   Hence
$L''$ has at least two positive tangles and at least two negative
  tangles.  It follows that the reduced form for $L''$ is
  non-alternating and adequate.
\end{proof}

For a rational tangle $t=a_1\ldots a_m$ as above, let
$\overline{t}=\overline{a_1}\overline{a_2}\ldots\overline{a_m}$ 
denote its reflection.

\begin{lemma} 
\label{lem:ref-sym}
 Let $L=M(e;t_1,\ldots,t_p)$ 
be a Montesinos link and $L^r=M(-e;\overline{t_1},\ldots,\overline{t_p})$ denote its reflection. Then $\varepsilon(L^r)=-\varepsilon(L)-p$. 
\end{lemma}
\begin{proof} The continued fraction expansion of $t$ implies that 
the reflection of $t$, $\overline{t}=-t$. It follows that 
$\bfloor{1/\, \overline{t}}=\bfloor{1/-t}=-\bfloor{1/t}-1$.
Hence 
$$\varepsilon(L^r)= -e+\sum_{i=1}^p \bfloor{1/\, \overline{t_i}} = -e+\sum_{i=1}^p \left( -\bfloor{1/t_i}-1 \right) 
=\left(-e-\sum_{i=1}^p \bfloor{1/t_i}\right) -p = -\varepsilon(L)-p. \qedhere$$
\end{proof}

\section{Determinant of Montesinos links} 

The determinant of rational and Montesinos links follows directly from
Conway's \emph{determinant fraction} identities for tangle sum
$t_{a+b}$ and product $t_{ab}$ given in \cite{Conway}:
\begin{equation*}
\label{det-sum}
\frac{det(1* t_{a+b})}{det(1* t_{a+b}0)}=\frac{det(1* t_{a})}{det(1* t_{a}0)}+\frac{det(1* t_{b})}{det(1* t_{b}0)} \ \mathrm{and} \
\frac{det(1* t_{ab})}{det(1* t_{ab}0)}=\frac{det(1* t_{a}0)}{det(1* t_{a})}+\frac{det(1* t_{b})}{det(1* t_{b}0)}.
\end{equation*}
where $det(K)$ is Conway's determinant. The usual deteminant $\det(K)
= |det(K)|$ (Section 7 in \cite{Conway}). We derive the formula for 
the determinant of Montesinos links (see also \cite{Asaeda}).

\begin{prop}
  $\displaystyle{\det\left(M\left(e;\frac{\alpha_1}{\beta_1},\ldots,\frac{\alpha_p}{\beta_p}\right)\right)=
\left|    \left( \prod_{i=1}^p \alpha_i \right) \left(e+\sum_{i=1}^p
      \frac{\beta_i}{\alpha_i} \right) \right |}$.
\end{prop}
\begin{proof}
  Let $t=\alpha/\beta$ be the rational tangle $a_1 a_2 \dots a_m$, as above. Then it follows by induction on $m$ and by the determinant fraction identity for the product that
$det(1*t)=\alpha$ and $det(1*t0)=\beta$. 

  Let
$\displaystyle{t_i=\frac{\alpha_i}{\beta_i}}$, so
$\displaystyle{\frac{det(1*t_i)}{det(1*t_i0)}=\frac{\alpha_i}{\beta_i}}$.
Since $1*(e+t_10+\ldots+t_p0)0= 1*t_1 \ \# \ldots \# \ 1*t_p$, 
$det(1*(e+t_10+\ldots+t_p0)0)=det(1*t_1) \times \ldots \times det(1*t_p) 
= \prod_{i=1}^p \alpha_i$. Using the determinant fraction identity for 
the sum we get 
\begin{align*}
\frac{det(1*(e+t_10+\ldots+t_p0))}{det(1*(e+t_10+\ldots+t_p0)0)}&=
\frac{det(1*e)}{det(1*e0)}+\sum_{i=1}^p \frac{det(1*t_i0)}{det(1*t_i00)},\\
\det(1*(e+t_10+\ldots+t_p0))&= \left| \left( \prod_{i=1}^p \alpha_i \right)  
\left(  e + \sum_{i=1}^p \frac{\beta_i}{\alpha_i} \right)\right|. \qedhere
\end{align*}
\end{proof}

\section{Quasi-alternating Montesinos links}

 \begin{prop} Let $s, r_i\neq 1$ be positive rational numbers for
  $i=1,\ldots, n$. Then the Montesinos link 
$M(0;r_1, \ldots, r_n, -s)$ is quasi-alternating if
 $s > \mathrm{min}\{r_1,\ldots, r_n\}$. The statement is true 
for any position of the tangle $-s$. 
\label{thm:e=0}
  \end{prop}
\begin{proof}
  We will prove the statement by induction on $n$. For $n=1$, Proposition \ref{prop:2tangle} implies that $M(0;r,-s)$ is a rational link, which is quasi-alternating when $r \neq s$ (for $r=s$, this gives the unlink on two components which is not quasi-alternating). This proves the base case.

Let $s > \mathrm{min}\{r_1,\ldots,r_n\}$. By the induction hypothesis
$M(0;r_1,\ldots, r_n, -s)$ is quasi-alternating. Let $L$ be the
diagram $M(0;r_1, \ldots, r_n, 1,-s)$ and let $c$ be the single
crossing to the left of the $-s$ tangle. $L_{\infty}=M(0;r_1,\ldots,
r_n, -s)$.  $L_0$ naturally splits as a connect sum of horizontal
closure of tangles of the type $t0$, for a rational tangle $t$. 
Since the horizontal closure of $t0$ is isotopic
to $1*\overline{t}$, we have $L_0=1*\overline{r_1} \# 1*\overline{r_2}\# \ldots \# 1*\overline{r_n} \# 1*s$. See Figure \ref{fig:connectsum}.
\begin{figure}[htbp]
\begin{center}
\includegraphics[height=2in]{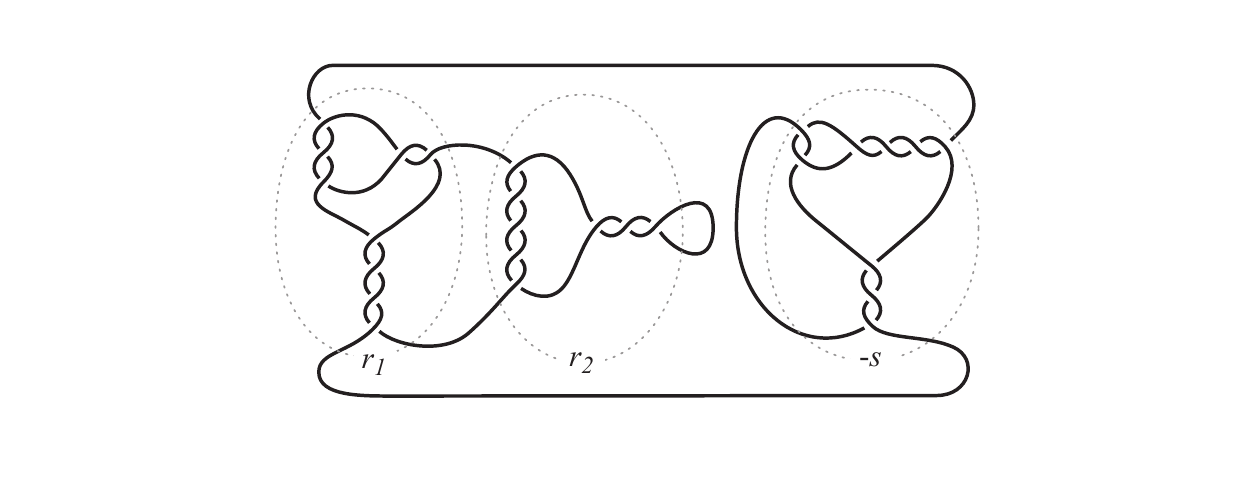}
\caption{The link $L_0$ is a connect sum of horizontal closures of rational tangles.}
\label{fig:connectsum}
\end{center}
\end{figure}
Let $r_i = \alpha_i/\beta_i$ and $s=\alpha/\beta$,
where all the $\alpha$'s and $\beta$'s are positive integers.  By
the formula for the determinant of rational and Montesinos links
\begin{align*}
\mathrm{det}(L_0)= \mathrm{det}(1*s) \left (\prod_{i=1}^n \mathrm{det}(1*\overline{r_i}) \right ) = \left|\alpha \prod_{i=1}^n \alpha_i\ \right| ,\  \  
\mathrm{det}(L_{\infty})= \left| \alpha \prod_{i=1}^n \alpha_i \left (
\sum_{i=1}^n \frac{\beta_i}{\alpha_i}-\frac{\beta}{\alpha} \right ) 
\right |,
\end{align*}
and $\det(L_\infty)\neq 0$ because $L_\infty = M(0;r_1,\dots,r_n,-s)$ is assumed to be quasi-alternating. 

Since $s > \mathrm{min}\{r_1,\ldots,r_n\}$, we have
$\displaystyle{ \left( \sum_{i=1}^n \frac{1}{r_i}-\frac{1}{s} \right ) 
= \left( \sum_{i=1}^n \frac{\beta_i}{\alpha_i}-\frac{\beta}{\alpha} \right )
>0}$. 
Hence, 
\begin{align*}
\mathrm{det}(L_0)+\mathrm{det}(L_{\infty})
= \alpha \prod_{i=1}^n \alpha_i \left (1+
\sum_{i=1}^n \frac{\beta_i}{\alpha_i}-\frac{\beta}{\alpha}  \right )
= \mathrm{det}(L).
\end{align*}

$L_0$ is quasi-alternating by Lemma 2.3 in \cite{Champanerkar} and
$L_{\infty}$ is quasi-alternating by the induction hypothesis, hence $L$
is quasi-alternating at the crossing $c$. Using Theorem 2.1 in
\cite{Champanerkar}, we can extend $c$ to a rational tangle. This shows that if $s >
\mathrm{min}\{r_1,\ldots, r_{n+1}\}$ then $M(0;r_1,\ldots, r_{n+1},
-s)$ is quasi-alternating. Since we did not use the position of the
tangle $-s$ in the argument, the same argument works for any position
of the tangle $-s$.
\end{proof}

\begin{remark} 
Unlike the statement of Proposition 2.2 in \cite{Greene}, the condition $s>\min\{r_1,\dots,r_n\}$ appearing above is not a necessary condition. 
For example, 
$$ M(0;2,7,-4)= M(-1;2,7,4/3)=M(0;2,-7/6,4/3).$$
The leftmost diagram satisfies the condition of Proposition \ref{thm:e=0}, 
and, hence, it is quasi-alternating. However, the rightmost diagram 
fails to satisfy the above condition. 
\end{remark}

We will use Proposition \ref{thm:e=0} to prove a sufficient condition 
for any Montesinos link to be quasi-alternating, in terms of the invariant 
$\varepsilon$ defined above. 
Recall that for 
$0\neq t=\displaystyle{\alpha/\beta} \in \Q$, 
$$\widehat{t}=\frac{1}{\{ \frac{1}{t} \} },\quad     
t^f=\frac{\alpha}{\beta- \alpha} \mathrm{\ if\ } t >0,\quad  
t^f=\frac{\alpha}{\beta + \alpha} \mathrm{\ if\ } t <0,\quad   
\varepsilon= e + \sum_{i=1}^p  \left\lfloor\frac{1}{t_i}\right\rfloor.  $$

\begin{theorem} 
\label{thm:qam}
Let $L=M(e;t_1, \ldots, t_p)$ be a Montesinos link.
Then $L$ is quasi-alternating if 
  \begin{enumerate}
  \item $\varepsilon > -1$, or
\item $\varepsilon=-1$ and $|\widehat{t_{i}}^f| > \widehat{t_j}$ for some $i \neq j$, or 
\item $\varepsilon < 1-p$, or 
\item $\varepsilon=1-p$ and $|\widehat{t_{i}}^f| < \widehat{t_j}$ for some $i \neq j$. 
  \end{enumerate}
\end{theorem}
\begin{proof} 
  Since reflections of quasi-alternating links are quasi-alternating,
  it is enough to consider $L$ or its reflection.  By the symmetry of
  $\varepsilon$ under reflections proved in Lemma \ref{lem:ref-sym}, it
  suffices to consider the case when $\varepsilon \geq -p/2$. Cases (3) 
and (4) follow from cases (1) and (2) respectively.  

If $\varepsilon > -1$ then, by Proposition \ref{prop:adq-alt}, $L$ has an 
alternating diagram, and hence it is quasi-alternating. 

If $\varepsilon=-1$ then $L=M(-1;\widehat{t_1},\ldots,\widehat{t_p})$. The
condition $|\widehat{t_{i}}^f| > \widehat{t_{j}}$ for some $i \neq
j$ implies that we can use a positive flype on the tangle
$\widehat{t_{i}}$ to convert $L$ to an equivalent link which satisfies the
condition in Proposition \ref{thm:e=0}.
\end{proof}
\subsection{Examples}
\begin{enumerate}
\item $M(3;\frac{31}{7},\frac{5}{16},\frac{-29}{9})$,
  $\varepsilon = 3 + \lfloor\frac{7}{31}\rfloor +
    \lfloor\frac{16}{5}\rfloor +
    \lfloor\frac{9}{-29}\rfloor = 3 + 0 + 3 -1 = 5 > -1$
  which is quasi-alternating by case 1 of Theorem \ref{thm:qam}. \\

\item $M(-1; \frac{3}{2}, \frac{4}{3}, \frac{7}{4})$, which is in reduced form; i.e., $\widehat{t_i}=t_i$.
  $|\widehat{t_1}^f|=\frac{3}{1}$,
  $|\widehat{t_2}^f|=\frac{4}{1}$,
  $|\widehat{t_3}^f|=\frac{7}{3}$.
Since $|\widehat{t_1}^f|>\widehat{t_2}$, this link is
quasi-alternating by case 2 of Theorem \ref{thm:qam}.
In particular, $M(-1; \frac{3}{2}, \frac{4}{3}, \frac{7}{4}) = M(0;\frac{-3}{1},\frac{4}{3}, \frac{7}{4})$, by applying a positive flype 
to the first tangle. The resulting link is quasi-alternating 
by Proposition \ref{thm:e=0}.
\end{enumerate}

\section{Non-quasi-alternating Montesinos links}

\begin{theorem} 
Let $L=M(e;t_1, \ldots, t_p)$ be a Montesinos link with $p \geq 3$.
Then $L$ is non-quasi-alternating if 
  \begin{enumerate}
  \item $1-p < \varepsilon < -1$, or 
\item $\varepsilon=-1$ and $\widehat{t_i}>2$ for all $i=1,\dots, p$, or 
\item $\varepsilon=1-p$ and $|\widehat{t_i}^f|>2$ for all $i=1,\dots, p$.
  \end{enumerate}
\label{thm:nqam}
\end{theorem} 

\begin{proof}
Case (1) implies that $|\varepsilon+p/2|<p/2-1$ and, by Proposition \ref{prop:adq-alt}, $L$ has a non-alternating and adequate diagram. The Khovanov homology of a link $L$ with such a diagram is thick \cite{Khovanov}, which implies that $L$ is not quasi-alternating \cite{Manolescu}.

For case (2), assume that $L$ is in reduced form with $\varepsilon=-1$ and $t_i=\widehat{t_i}>2$ for all $i=1,\ldots,p$. 
We will show that the double branched cover $\Sigma(L)$ is not an $L$-space.
A closed, connected 3-manifold $Y$ is an \emph{$L$-space} if it is a rational homology sphere with the property that the rank of its Heegaard Floer homology $HF(Y)$ equals $|H_1(Y;\Z)|$.
Recall that the branched double cover $\Sigma(L)$ of a Montesinos link $L$ is the orientable Seifert fibered space 
$S(0;\varepsilon,t_1,\dots,t_p)$ with base orbifold $S^2$ \cite{Montesinos}.
The manifold $\Sigma(L)$ is a rational homology sphere iff $\det(L)\neq 0$ 
(see for example \cite{lickorish}).
If $\det(L)=0$ then $L$ is non-quasi-alternating.
Otherwise, the following theorem provides the $L$-space obstruction. 
First, define a group $G$ to be \emph{left-orderable} if there exists a left invariant strict total ordering on $G$. 

\begin{theorem}[\cite{Boyer-Gordon-Watson}] A closed connected Seifert fibered 3-manifold $X$ is not an $L$-space iff $\pi_1(X)$ is left-orderable.
\label{thm:ls}\end{theorem}

\noindent The next result offers the exact conditions for an orientable Seifert fibered space to have a left-orderable fundamental group.

\begin{theorem}[\cite{Boyer-Rolfsen-Wiest}] Let $X$ be an orientable Seifert fibered 3-manifold which is a rational homology sphere. Then $\pi_1(X)$ is left-orderable iff $\pi_1(X)$ is infinite, the base orbifold of $X$ is the 2-sphere with cone points, and $X$ admits a horizontal foliation.
\label{thm:lo}\end{theorem}

\noindent The fundamental group $\pi_1(\Sigma(L))$ is infinite if $\Sigma_{i=1}^p 1/\alpha_i \leq p-2$, where $t_i=\alpha_i/\beta_i$ \cite{Burde-Zieschang}.
This condition is satisfied for $p\geq 3$ and $t_i>2$. 
Thus it remains to show that the space $\Sigma(L)$ admits a \emph{horizontal foliation}; i.e., a foliation which is everywhere transverse to the Seifert fibers.
The following result provides the conditions under which a Seifert fibered space admits a horizontal foliation.

\begin{theorem}[\cite{Eisenbud-Hirsch-Neumann},\cite{Jankins-Neumann},\cite{Naimi}] Let $S=S(0;-1, \alpha_1/\beta_1,\dots, \alpha_n/\beta_n)$ be an orientable Seifert fibered 3-manifold, where $n \geq 3$ and $\alpha_i/\beta_i> 1$ are rational numbers. Then $S$ admits a horizontal foliation iff there exist relatively prime integers $0 < a < m$ such that 
$$
\frac{\alpha_{\sigma(1)}}{\beta_{\sigma(1)}} >\frac{m}{a}, \quad
\frac{\alpha_{\sigma(2)}}{\beta_{\sigma(2)}} >\frac{m}{m-a}, \quad
\frac{\alpha_{\sigma(i)}}{\beta_{\sigma(i)}} >m, 
$$ 
where $3\leq i \leq n$ and $\sigma$ is a permutation of $\{1,2,...,n\}$.
\label{thm:hf}\end{theorem}

\noindent Given $\varepsilon=-1$ and $t_i>2$ for all $i=1,\dots, p$, Theorem \ref{thm:hf} implies that, for the choice of $m=2$ and $a=1$, the Seifert fibered space $\Sigma(L)$ admits a horizontal foliation. 
The fundamental group $\pi_1(\Sigma(L))$ is left-orderable according to Theorem \ref{thm:lo}. 
Finally, $\Sigma(L)$ is not an L-space by Theorem \ref{thm:ls}, and, therefore, $L$ is non-quasi-alternating.

For case (3), assume that $L$ is in reduced form with $\varepsilon =
1-p$ and $|\widehat{t_i}^f| = |t_i^f| > 2$ for all $i=1,\ldots,p$. 
The reflection $L^r=M(p-1;\overline{t_1}, \ldots, \overline{t_p})$. Note that,
\begin{equation}
 t=\frac{\alpha}{\beta}>1,\ \ \ \  \overline{t}=\frac{-\alpha}{\beta},\ \ \ \  
\overline{t}^f=\frac{\alpha}{\alpha - \beta} = |t^f| > 1. 
\label{eqn:tbarf}
\end{equation}
These relations together with Lemma \ref{lem:flype} imply that the application of $p$ negative flypes on $L^r$ yields
$M(-1;|t_1^f|, \ldots, |t_p^f|)$, which is in reduced form by Equation 
(\ref{eqn:tbarf}).  Case (3) now follows from case (2) and the fact that
reflections of quasi-alternating links are quasi-alternating.
\end{proof}

\begin{remark}
There are more families of non-quasi-alternating Montesinos links accessible by the proof of Theorem \ref{thm:nqam}.  For example, by choosing $m=3$ and $a=2$ in the notation of Theorem \ref{thm:hf}, one easily shows that if $\widehat{t_1}>3/2$ and $\widehat{t_i} > 3$ for $i\geq 2$, then $L$ is not quasi-alternating.  
In fact, for any choice of $m$ with $a=m-1$, the link is non-quasi-alternating for $\widehat{t_1}>m/(m-1)$ and $\widehat{t_i}>m$, $i\geq 2$. 
In general, any Montesinos link whose parameters satisfy the hypothesis of Theorem \ref{thm:hf} is non-quasi-alternating by the same argument.
\end{remark}

The proof of Theorem \ref{thm:nqam} offers an alternative to the
4-manifold techniques Greene used to establish obstructions to the
quasi-alternating property of pretzel links.  A key step in the
classification of quasi-alternating pretzel links is Proposition 2.2
in \cite{Greene}, which states that the pretzel
$P(0;p_1,\dots,p_n,-q)$ is quasi-alternating iff
$q>\min\{p_1,\dots,p_n\}$, where $n\geq 2$, $p_1,\dots,p_n\geq 2$, and
$q\geq 1$.  We obtain an alternative obstruction in most cases.

\begin{prop}
For the same conditions above, the pretzel $P(0;p_1,...,p_n,-q)$ is non-quasi-alternating if $q+1<\min\{p_1,\ldots, p_n\}$.
\end{prop}

\begin{proof} 
If $q=1$, the pretzel $P(0;p_1,...,p_n,-q)$ is equivalent to the reduced Montesinos link $M(-1; p_1,\dots,p_n)$ with $2 <
 \min\{p_1,\dots,p_n\}$.  
According to case (2) of Theorem \ref{thm:nqam}, the links are non-quasi-alternating.  
If $q>1$, then the pretzel is equivalent to the reduced Montesinos link $M(-1;p_1,\dots,p_n,q/(q-1))$.  
The fact that $q/(q-1) > (q+1)/q$ implies that choosing $m=q+1$ and $a=q$, in the notation of Theorem \ref{thm:hf}, demonstrates that the branched double cover admits a horizontal foliation.
\end{proof}

\subsection*{6.1 Further questions and examples}

It is a natural question to ask whether the condition given in Theorem 5.3 is
necessary. Indeed Qazaqzeh, Chbili, and Qublan have asserted the following:

\begin{conjecture}[\cite{Khaled}]
A Montesinos link is quasi-alternating if and only if it satisfies the conditions of Theorem \ref{thm:qam}.
\end{conjecture}

Theorem \ref{thm:nqam} partially resolves the conjecture. 
It remains to investigate Montesinos links whose parameters satisfy neither the conditions of Theorem \ref{thm:qam} nor the conditions for admitting a horizontal foliation given in Theorem \ref{thm:hf}.  
These are Montesinos links that may be non-quasi-alternating but whose double branched covers are L-spaces. 
Below we discuss several families of such links. \\

\begin{enumerate}

\item A Montesinos link $M(-1;t_1,t_2,\ldots, t_n)$ in reduced form and such that $|t_i^f|= t_j$, where $t_i$ and $t_j$ are the least and second least among the parameters $t_1,t_2,\ldots, t_n$.
Greene's first example of a non-quasi-alternating knot with thin homology, $11n50 = M(-1;5/2,3,5/3)$, is an example of such a link.  
In the preprint, he remarks that his proof generalizes to show that the infinite family $M(0;(m^2+1)/m,n,-(m^2+1)/m)= M(-1;(m^2+1)/m,n,(m^2+1)/(m^2-m+1))$ for positive integers $m, n\geq 2$ is non-quasi-alternating \cite{Greene}.\\

\item A pretzel link
$P(0;p_1,...,p_n,-q)=M(-1;p_1,\dots,p_n,q/(q-1))$ 
that satisfies the condition $q=\min\{p_1,\ldots, p_n\}$.
Any such link is known to be non-quasi-alternating \cite{Greene}. \\

\item A pretzel link 
$P(0;p_1,...,p_n,-q)=M(-1;p_1,\dots,p_n,q/(q-1))$, 
for which $q+1=\min\{p_1,\ldots, p_n\}$ and $p_i = q+1$ for all $i$.
However, if $p_i$ exceeds the numerator of a rational number between $q$ and $q+1$ for all $i$ except one, then the link will satisfy the conditions of Theorem \ref{thm:hf}. 
In general, the pretzels $P(0;p_1,...,p_n,-q)=M(-1;p_1,\dots,p_n,q/(q-1))$ 
for which $q+1=\min\{p_1,\ldots, p_n\}$ is known to be non-quasi-alternating \cite{Greene}.  
\\

\item The pretzel $P(0; 3, 3, 3, -2) = 11n81$ is such an example, and it has thick Khovanov homology. 
Since adding rational tangles preserves the width of Khovanov homology (\cite{lowrance, watson-1}), one may obtain infinite families of Montesinos links which do not satisfy either conditions. 
Having thick Khovanov homology, these are non-quasi-alternating. \\

\item Watson pointed us to another family of the form 
$M(0;(2n+1)/2, n+1, (-2n-1)/2 ) = M(-1;(2n+1)/2, n+1, (2n+2)/(2n-1) )$, where $n\geq 2$. 
See Figure \ref{fig:liam-examples}. 
Their quasi-alternating status is undetermined. 

\end{enumerate}

\begin{figure}[h]
\begin{center}
\includegraphics[height=1.5in]{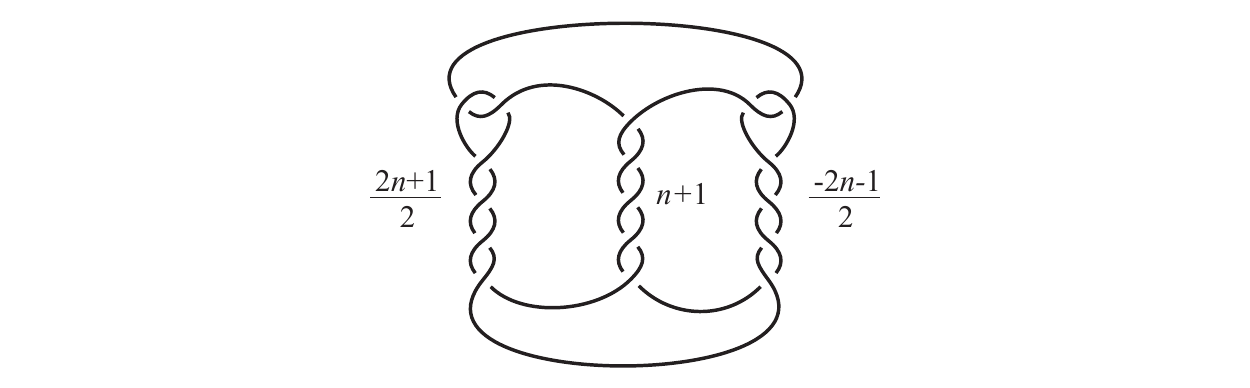}
\caption{Montesinos links $M(0;(2n+1)/2, n+1, (-2n-1)/2 )$ for $n\geq2$ do not satify the conditions of
  Theorem \ref{thm:qam} and Theorem \ref{thm:hf}. 
  }
\label{fig:liam-examples}
\end{center}
\end{figure}

\bibliographystyle{plain}
\bibliography{references}

\begin{thebibliography}{10}

\bibitem{Asaeda}
Marta~M. Asaeda, J{\'o}zef~H. Przytycki, and Adam~S. Sikora.
\newblock Kauffman-{H}arary conjecture holds for {M}ontesinos knots.
\newblock {\em J. Knot Theory Ramifications}, 13(4):467--477, 2004.

\bibitem{Bonahon}
Francis Bonahon.
\newblock {\em Involutions et fibr\'es de Seifert dans les vari\'et\'es de
  dimension 3}.
\newblock 1979.
\newblock Th\'ese de 3e cycle Orsay.

\bibitem{Boyer-Gordon-Watson}
Steven Boyer, Cameron~McA. Gordon, and Liam Watson.
\newblock On l-spaces and left-orderable fundamental groups.
\newblock {\em arXiv/1107.5016}, 2011.

\bibitem{Boyer-Rolfsen-Wiest}
Steven Boyer, Dale Rolfsen, and Bert Wiest.
\newblock Orderable 3-manifold groups.
\newblock {\em Ann. Inst. Fourier (Grenoble)}, 55(1):243--288, 2005.

\bibitem{Burde-Zieschang}
Gerhard Burde and Heiner Zieschang.
\newblock {\em Knots}, volume~5 of {\em de Gruyter Studies in Mathematics}.
\newblock Walter de Gruyter \& Co., Berlin, second edition, 2003.

\bibitem{Champanerkar}
Abhijit Champanerkar and Ilya Kofman.
\newblock Twisting quasi-alternating links.
\newblock {\em Proc. Amer. Math. Soc.}, 137(7):2451--2458, 2009.

\bibitem{Conway}
John~H. Conway.
\newblock An enumeration of knots and links, and some of their algebraic
  properties.
\newblock In {\em Computational {P}roblems in {A}bstract {A}lgebra ({P}roc.
  {C}onf., {O}xford, 1967)}, pages 329--358. Pergamon, Oxford, 1970.

\bibitem{Eisenbud-Hirsch-Neumann}
David Eisenbud, Ulrich Hirsch, and Walter Neumann.
\newblock Transverse foliations of {S}eifert bundles and self-homeomorphism of
  the circle.
\newblock {\em Comment. Math. Helv.}, 56(4):638--660, 1981.

\bibitem{Greene}
Joshua Greene.
\newblock Homologically thin, non-quasi-alternating links.
\newblock {\em Math. Res. Lett.}, 17(1):39--49, 2010.

\bibitem{Jankins-Neumann}
Mark Jankins and Walter~D. Neumann.
\newblock Rotation numbers of products of circle homeomorphisms.
\newblock {\em Math. Ann.}, 271(3):381--400, 1985.

\bibitem{Khovanov}
Mikhail Khovanov.
\newblock Patterns in knot cohomology. {I}.
\newblock {\em Experiment. Math.}, 12(3):365--374, 2003.

\bibitem{lickorish}
W.~B.~Raymond Lickorish.
\newblock {\em An introduction to knot theory}, volume 175 of {\em Graduate
  Texts in Mathematics}.
\newblock Springer-Verlag, New York, 1997.

\bibitem{lowrance}
Adam Lowrance.
\newblock The {K}hovanov width of twisted links and closed 3-braids.
\newblock {\em Comment. Math. Helv.}, 86(3):675--706, 2011.

\bibitem{Manolescu}
Ciprian Manolescu and Peter Ozsv{\'a}th.
\newblock On the {K}hovanov and knot {F}loer homologies of quasi-alternating
  links.
\newblock In {\em Proceedings of {G}\"okova {G}eometry-{T}opology {C}onference
  2007}, pages 60--81. G\"okova Geometry/Topology Conference (GGT), G\"okova,
  2008.

\bibitem{Montesinos}
Jos{\'e}~M. Montesinos.
\newblock Seifert manifolds that are ramified two-sheeted cyclic coverings.
\newblock {\em Bol. Soc. Mat. Mexicana (2)}, 18:1--32, 1973.

\bibitem{Naimi}
Ramin Naimi.
\newblock Foliations transverse to fibers of {S}eifert manifolds.
\newblock {\em Comment. Math. Helv.}, 69(1):155--162, 1994.

\bibitem{Osvath-Szabo:DoubleCovers}
P.~Ozsv{\'a}th and Z.~Szab{\'o}.
\newblock On the {H}eegaard {F}loer homology of branched double-covers.
\newblock {\em Adv. Math.}, 194, 2005.

\bibitem{Khaled}
Khaled Qazaqzeh, Nafaa Chbili, and Balkees Qublan.
\newblock Characterization of quasi-alternating montesinos links.
\newblock {\em arXiv/1205.4650}, 2012.

\bibitem{Watson}
Liam Watson.
\newblock A surgical perspective on quasi-alternating links.
\newblock In {\em Low-dimensional and symplectic topology}, volume~82 of {\em
  Proc. Sympos. Pure Math.}, pages 39--51. Amer. Math. Soc., Providence, RI,
  2011.

\bibitem{watson-1}
Liam Watson.
\newblock Surgery obstructions from {K}hovanov homology.
\newblock {\em Selecta Math. (N.S.)}, 18(2):417--472, 2012.

\bibitem{Widmer}
Tamara Widmer.
\newblock Quasi-alternating {M}ontesinos links.
\newblock {\em J. Knot Theory Ramifications}, 18(10):1459--1469, 2009.

\end{thebibliography}

\end{document}